\documentclass[12pt]{amsart}

\usepackage{amsthm}
\usepackage{amssymb}
\usepackage{latexsym}
\usepackage{url}
\usepackage{amsmath}
\usepackage[dvips]{graphics}
\usepackage{graphicx}
\usepackage{epsf}
\usepackage{verbatim}
\usepackage{geometry}
\usepackage{hyperref}
\hypersetup{breaklinks=true,
pagecolor=white,
colorlinks=true,
citecolor=black,%
filecolor=black,%
linkcolor=black,%
urlcolor=black
}

\newtheorem{thm}{Theorem}[section]

\newtheorem{lem}[thm]{Lemma}

\address{Department of Mathematics, The Technion - Israel Institute of 
Technology, 32000 Haifa, Israel. Current  address (July 2011): Department of Mathematics, University of Haifa, Mount Carmel, 31905 Haifa, Israel. }
\email{dream@tx.technion.ac.il}

\keywords{Analytic function, fixed point, fundamental theorem of algebra, FTA, open mapping, OMT,  polynomial.}

\subjclass[2000]{26C10, 30A99, 47H10}

\begin{document}

\title[The open mapping theorem and the FTA]{The open mapping theorem  and the fundamental theorem of algebra}
\author{Daniel Reem}
\date{}
\maketitle

\begin{abstract}
This note is devoted to two classical theorems:
the open mapping theorem for analytic functions (OMT) and the
fundamental theorem of algebra (FTA). We present a new proof of
the first theorem, and then derive the second one by a simple
topological argument. The proof is elementary in nature and does
not use any kind of integration (neither complex nor real). In
addition, it is also independent of the fact that the roots of an
analytic function are isolated. The proof is based on either the
Banach or Brouwer fixed point theorems. In particular, this shows
that one can obtain a proof of the FTA (albeit indirect) which is
based on the Brouwer fixed point theorem, an aim which was not
reached in the past
and later the possibility to achieve it was questioned.
We close this note with a simple generalization of the FTA.
A short review of certain issues related to the OMT and the FTA is also included. \end{abstract}

\section{Introduction}
The open mapping theorem for analytic functions (OMT) says that any
non-locally constant analytic function $f$ is open (strongly
interior), i.e., $f(V)$ is open whenever $V$ is open. The usual
proofs of this theorem are based on either  Rouch\'{e}'s theorem
\cite[pp. 306-307]{LevinsonRedheffer},\cite[Chapter V]{Stoilow}, on the
index (winding number) and the argument principle
\cite[p. 173]{Ahlfors},
or on simple versions of the maximum or minimum modulus principles \cite[pp. 172-173]{Burckel},
\cite[pp. 256-257]{Remmert}.
 All of these proofs are based on complex integration theory.

Since the theorem ``is topological in nature'', mathematicians
sought a proof which is topological in nature, and uses minimal
analytical theory. In particular, it should not use the ``usual
sophisticated machinery of analysis'' \cite{WhyburnReview}, such as development in
power series or integration. In various places \cite[p.
260]{EgglestonUrsell},\cite[p. 93] {TitusYoung},\cite[p.
2]{WhyburnMemoir},\cite[p. viii]{WhyburnBook} such a proof is
regarded as an elementary, and this philosophy is well expressed
in the book of G. T. Whyburn \cite{WhyburnBook}.

In 1952, it seemed that such a proof was found by H. G. Eggleston
and H. D. Ursell \cite{EgglestonUrsell}. This was mentioned by
Eggleston and Ursell, and also by Whyburn in his review
\cite{WhyburnReview}. The proof is elementary in the sense that one
uses only the fact that $f'$ exists, and does not use the
development of $f$ in power series, or apply integration directly to
it. However, a careful reading of the proof shows that it is based
on an integral definition of the index, and also on the properties
of the complex exponent function $t\mapsto \exp(it)$ and the number
$\pi$ (see below), and hence it does use the usual machinery of
analysis.

Another attempt to find ``an elementary'' proof was by C. J. Titus
and G. S. Young \cite{TitusYoung}. They established a theorem
which is elementary in the above sense and implies the OMT,
assuming one knows that $f$ is light, i.e., that for any $c$, the
roots of the equation $f(z)=c$ are isolated. They however could
not establish the lightness of $f$ (in an
 elementary way).

Later, by modifications of the arguments in \cite{EgglestonUrsell},
Whyburn \cite[p. 76]{WhyburnBook} obtained a proof which bypassed
the integral definition of the index and seemed apparently
elementary, but again, a careful reading of it shows that it is
still based on either integration or power series, since it is based
on the definition and properties of the complex exponent function
$t\mapsto \exp(it)$ and the number $\pi$. (See \cite[pp.
43-46]{Ahlfors} for  discussion and formal definition of them using
power series. The definition $\exp(it)=\cos(t)+i\sin(t)$ which is
used in \cite[p. 53]{WhyburnBook} can be considered as formal only
after one gives a formal non-geometric definition of
$\cos(t),\sin(t)$ and $\pi$, and proves some basic trigonometric
formulas. See \cite[pp. 432-438]{Kitchen} where this is done by use of
integrals.)

 We note in addition that
all of these proofs are not elementary in the usual sense, for example
because they are long and based on hard theorems, such as the Jordan
curve theorem or considerations from degree theory. Especially this
is true for that of Whyburn, which is much longer and based on
results of several chapters of his
book.\\

Here we present a new proof of the OMT which is elementary in the
usual sense of the word (short, not based on hard theorems or hard
arguments), and is not based on complex integration theory. In fact,
it is not based on integration of any kind. Another property of this
proof is that it does not use the fact that $f$ is light. As far as
we know, all known proofs of the OMT, and in particular those
mentioned in the references, are based on this fact. Our proof is
based on either the Brouwer or the Banach fixed point theorems (so
in fact it can be considered as elementary only in the latter case).
Since the proof is also based on the theory of power series
\cite[pp. 39-46]{Ahlfors}, and on the Weierstrass definition of
analytic function as one which can be locally developed in power
series, it is not elementary in the sense described earlier.

As a corollary of this theorem  we obtain the fundamental theorem of
algebra (FTA) by a simple topological argument ($\mathbb{C}$ is
connected). In particular, this shows that one can obtain a proof
(albeit indirect) of the FTA based on the Brouwer fixed point
theorem. This is interesting, because in the past there was an
attempt to do it (by B. H. Arnold \cite{Arnold}) but this attempt
failed \cite{ArnoldMistake} due to a serious mistake. Moreover,
about 35 years later it was shown by A. Aleman \cite{Aleman} that it is impossible
to prove the FTA by the Brouwer fixed point theorem if one tries to
use the methods of \cite{Arnold}, i.e., that there is no hope to
correct the mistake in \cite{Arnold}. This cast doubt on the
possibility of proving the FTA by applying the Brouwer fixed point theorem.

It is interesting to note in this connection that there does exist a
proof of the FTA which is based on a fixed point theorem (the
Lefschetz fixed point theorem; see \cite{Medeiros}). In addition, a
careful reading of the proof of the FTA given in \cite{Wolfenstein}
shows that one of its ingredients is an argument similar to the one
appearing in the proof of the Banach fixed point theorem.

We also note that the connection between the OMT and the FTA is
not new. For example, R. L. Thompson \cite{Thompson} proved that
the FTA is equivalent to the open mapping theorem for polynomials.
In addition, the elementary proof of the FTA by S. Wolfenstein
\cite{Wolfenstein}, and the proof of the generalization of the FTA
by M. Reichaw (Reichbach) \cite[p. 160]{Reichaw} are also related
to connectedness and open mappings, but their arguments are
different from ours. (They use the fact that $f$ is locally open when
$f'(x)\neq 0$, that $f'(x)=0$ only on a finite set of points $A$
when $f$ is a polynomial, and that $\mathbb{C}\backslash A$ is
connected. In comparison, in our proof  we merely use the fact
that $\mathbb{C}$ is connected and it is irrelevant whether
$f'(x)$ vanishes.)

Finally, we note that we were informed about two things related to our proofs.
First, S. Reich has told us that the
topological argument we use for proving the FTA was also
independently mentioned by him in \cite{ReichAlgebra}, as a remark
on Thompson's paper. Second, R. B. Burckel has told us  that there is another elementary proof of the OMT, due to F. S. Cater \cite{Cater}. As mentioned in \cite{Cater}, there is another elementary proof of the OMT in the book of S. Lang \cite{Lang}. These two proofs are also based on power series and do not use integration of any kind. However, the arguments in both proofs are different from ours.
\section{proof of the OMT and the FTA}
We need the following simple lemma. It can be easily proved (and
improved) by use of integrals (
$T(y)-T(x)=\int_{[x,y]}T'(z)\textrm{d}z$), but also without
integrals as below.
\begin{lem}\label{lem:Inequality}
Let $B\subseteq \mathbb{C}$ be nonempty and convex. If $T:B\to
\mathbb{C}$ is
 differentiable and $\sup_{\xi\in B}|T'(\xi)|\leq a$,
then $T$ is Lipschitz on $B$ with a Lipschitz constant not greater
than $\sqrt{2}a$.
\end{lem}
\begin{proof}
We can write $T=u+iv$ where $u,v:B\to \mathbb{R}$ are
differentiable. Let $\xi_1,\xi_2\in B$ and let
$\gamma(t)=\xi_1+(\xi_2-\xi_1)t,\,t\in [0,1]$ be the line segment
connecting them. Since $|u_x|^2+|u_y|^2=|T'|^2\leq a^2$ by the
Cauchy-Riemann equations, the real version of Lagrange's mean value
theorem and the Cauchy-Schwarz inequality imply that
\begin{equation*}
|u(\xi_1)-u(\xi_2)|=|u(\gamma(0))-u(\gamma(1))|=|(u(\gamma))'(t)|\leq\|\nabla
u(\gamma)\||\xi_1-\xi_2|\leq a |\xi_1-\xi_2|.
\end{equation*}
The same holds for $v$. Hence
\begin{equation*}
|T(\xi_1)-T(\xi_2)|=\sqrt{|u(\xi_1)-u(\xi_2)|^2+|v(\xi_1)-v(\xi_2)|^2}\leq
\sqrt{2}a|\xi_1-\xi_2|.
\end{equation*}
\end{proof}
\begin{thm}\label{thm:OpenMapping}
Let $f(z)$ be a non-locally-constant analytic function defined in an
open set $V\neq \emptyset$. Then $f$ is open, i.e., $f(U)$ is open whenever
$U\subseteq V$ is open.
\end{thm}
\begin{proof}
Suppose $\emptyset\neq U\subseteq V$ is open, and let $z_0\in U$ and
$w_0=f(z_0)\in f(U)$. It should be proved that there exists $r>0$ such
that the open ball $B(w_0,r)$ of radius $r$ and center $w_0$ is
contained in $f(U)$, i.$\,$e., that for any $w\in B(w_0,r)$, the
equation $w=f(z)$ has a solution $z\in U$.

Since $f$ is analytic, it can be represented as
$f(z)=\sum_{k=0}^{\infty}a_k(z-z_0)^k$ in a neighborhood of $z_0$.
Let $1\leq k\in \mathbb{N}$ be the minimal index for which the
coefficient $a_k$ does not vanish. Such $k$ exists, because
otherwise $f$ is locally constant in that neighborhood of $z_0$ (and
in fact globally constant if $V$ is also connected, by the
identity/uniqueness theorem).
We can write $f(z)=a_0+(z-z_0)^k(a_k+h(z-z_0))$, where $a_0=w_0$ and
\begin{equation*}
h(\xi)=\sum_{p=k+1}^{\infty}a_p\xi^{p-k}.
\end{equation*}
By the change of variables $\xi=z-z_0$, the equation $f(z)=w$
becomes
\begin{equation}\label{eq:xi^k}
\xi^k=\frac{w-w_0}{a_k+h(\xi)}.
\end{equation}
Now it is tempting to take root and transform this equation to a
fixed point equation and then to use a corresponding fixed point
theorem, but one should be careful, because there are several
candidates for the root, and each one of them has a line of discontinuity. Let
$g_1,g_2$ be the complex functions defined by
\begin{equation*}
g_1(\xi)=|\xi|^{\frac{1}{k}}\cdot
e^{i\frac{\textrm{arg}(\xi)}{k}},\quad
g_2(\xi)=|\xi|^\frac{1}{k}\cdot e^{i\frac{\textrm{Arg}(\xi)}{k}}.
\end{equation*}
Roughly speaking, $g_1$ and $g_2$ are ``$\,\xi^{1/k}\,$". Formally,
$g_1$ and $g_2$ are right inverses of the function $G(\xi)=\xi^k$,
i.e., $G(g_1(\xi))=G(g_2(\xi))=\xi$. The reversed equalities
$g_1(G(\xi))=\xi,\, g_2(G(\xi))=\xi$ are not necessarily true (take
$k=2$ and $\xi=-1$ for example). Because $0\leq
\textrm{arg}(\xi)<2\pi$ and $-\pi\leq \textrm{Arg}(\xi)<\pi$, $g_1$
and $g_2$ are continuously differentiable
 in
$\mathbb{C}\backslash([0,\infty)\times\{0\})$ and
$\mathbb{C}\backslash((-\infty,0]\times\{0\})$ respectively. In
other words, $g_1$ ($g_2$) is continuously differentiable at any
$\xi\neq 0$ with $\textrm{Arg}(\xi)\neq 0$ ($\textrm{Arg}(\xi)\neq -\pi$).

Since $\lim_{\xi\to 0} h(\xi)=0$, there exists $0<\rho$
sufficiently small such that $B[z_0,\rho]\subset U$, and such that $|a_k/2|<|h(\xi)+a_k|$ and
$\textrm{Arg}(a_k/(h(\xi)+a_k))\in(-\pi/8,\pi/8)$ for all $\xi$ in
the small closed ball $B[0,\rho]=\overline{B(0,\rho)}$.

Let
\begin{equation*}
0<r\leq \min(|a_k/2|{\rho}^k,(1/(2\alpha))^k),
\end{equation*}
where
\begin{equation*} \alpha=(1/k)\cdot
|2/a_k|^{3-\frac{1}{k}}\cdot (1+\sup_{|\xi|\leq \rho}|h'(\xi)|).
\end{equation*}
The reason for choosing these values will become clear in a
moment. We claim that the ball $B(w_0,r)$ is contained in $f(U)$.
To see this, fix $w\in B(w_0,r)$. It can be assumed that $w\neq
w_0$, because otherwise $w=f(z_0)$ and we are done. It suffices to
show that \eqref{eq:xi^k} has a
solution $\xi \in B[0,\rho]$.
 Consider the equation
\begin{equation}\label{eq:fixed}
\xi=g\left(\frac{w-w_0}{h(\xi)+a_k}\right)\equiv T(\xi).
\end{equation}
Here we take $g=g_1$ if $\textrm{Re}((w-w_0)/a_k)< 0$ and $g=g_2$
otherwise. By applying the function $G(\xi)=\xi^k$ to
\eqref{eq:fixed},  we see that any solution of \eqref{eq:fixed} is a
solution of \eqref{eq:xi^k} (the converse is not necessarily true
because it may happen that $g(G(\xi))\neq \xi$), so it suffices to
show that \eqref{eq:fixed} has a solution $\xi \in B[0,\rho]$.

$T$ is well defined in $B[0,\rho]$ by the choice of $\rho$.
In addition, it is continuously differentiable there,  because
$(w-w_0)/(h(\xi)+a_k)$ is outside the discontinuous set (ray) of
$g$. (For example, if $g=g_1$, then
$\textrm{Arg}((w-w_0)/(h(\xi)+a_k))=\textrm{Arg}(a_k/(h(\xi)+a_k))+\textrm{Arg}((w-w_0)/a_k)\notin
(-\pi/8,\pi/8)$ by the choice of $\rho$.)

By the choice of $\rho$ it follows that $|T(\xi)|\leq
|2r/a_k|^{\frac{1}{k}}$
for $\xi\in B[0,\rho]$, so $T(B[0,\rho])\subseteq B[0,\rho]$ by the
choice of $r$. Since $T$ is continuous, we can finish the proof by
applying the Brouwer fixed point theorem to \eqref{eq:fixed}.
However, we will show below that the more elementary Banach fixed
point theorem suffices for this purpose.
By the choice of $r$ and $\alpha$,
\begin{equation*}
|T'(\xi)|=\frac{|w-w_0|^{\frac{1}{k}}\cdot|h'(\xi)|}{k\cdot
|h(\xi)+a_k|^{3-\frac{1}{k}}} \leq
\alpha \cdot r^{\frac{1}{k}},\,\, \forall \xi\in B[0,\rho],
\end{equation*}
so $\sup_{\xi\in B[0,\rho]}|T'(\xi)|\leq0.5$, again by the choice of
$r$. Thus, by Lemma \ref{lem:Inequality}, $T$ is Lipschitz on
$B[0,\rho]$ with a Lipschitz constant not greater than
$0.5\sqrt{2}<1$. Since $B[0,\rho]$ is a complete metric space, the
Banach fixed point theorem implies that \eqref{eq:fixed} has a
(unique) solution $\xi\in B[0,\rho]$.
\end{proof}

\begin{thm}\label{thm:FTA}
Let $1\leq m\in \mathbb{N}$ and let $f:\mathbb{C}\to
\mathbb{C},\,f(z)=a_mz^m+\ldots+a_1z+a_0$, be a polynomial of degree
$m$ ($a_m\neq 0$). Then $f$ is surjective, i.e.,
$f(\mathbb{C})=\mathbb{C}$. In particular, $0\in f(\mathbb{C})$,
i.e., $f$ has a root.
\end{thm}
\begin{proof}
Since $f(\mathbb{C})$ is nonempty ($f(0)\in f(\mathbb{C})$) and open
(by Theorem \ref{thm:OpenMapping}), it suffices to show that it is
closed, because then, the fact that $\mathbb{C}$ is connected will
imply that $f(\mathbb{C})=\mathbb{C}$.

Let $w\in \mathbb{C}$ be given and let $(z_n)_{n=1}^{\infty}$ be a sequence of complex numbers with the property that $\lim_{n\to\infty}f(z_n)=w$. In particular,the sequence  $(f(z_n))_{n=1}^{\infty}$ is bounded.
 Hence the sequence $(z_n)_{n=1}^{\infty}$ is bounded,
 because otherwise
\begin{equation*}
\overline{\lim}_{n\to\infty}\,\,|f(z_n)|=\overline{\lim}_{n\to\infty}\,\,|(z_n)^m|\cdot|a_m+\frac{a_{m-1}}{z_n}+\ldots+\frac{a_0}{(z_n)^m}|=
\infty\cdot |a_m|=\infty.
 \end{equation*}
Thus $(z_n)_n$ has a convergent subsequence $z_{n_k}\to z\in
\mathbb{C}$, so $w=f(z)$ since $f$ continuous.
\end{proof}

\section{Concluding remarks}
We finish by remarking that the proof of Theorem \ref{thm:FTA} implies that 
Theorem \ref{thm:FTA} can be obviously generalized as follows: an entire
analytic function $f$ is surjective if and only if its image is
closed. It would be interesting to obtain a simple necessarily
and/or sufficient condition for this in terms of the coefficients of
$f$. Unfortunately, with the exception of the condition that almost
all its coefficients vanish, i.e., that it is a polynomial, we do
not know any such condition.

Nevertheless, one can indeed obtain some simple sufficient
conditions for $f$ to be surjective in terms of a possible
representation that it might have. A trivial example is when $f$ is a
composition of two surjective mappings. Another example is when
$f(z)=p(g(z))+q(1/g(z))$, where $p$ and $q$ are non-constant
polynomials and $g$ is a non-constant analytic function which
does not vanish, such as $g(z)=\exp(\alpha z),\,\,\alpha\neq 0$.
Indeed, suppose $p(z)=\sum_{k=0}^n a_kz^k,\,q(z)=\sum_{k=0}^m
b_kz^k$ where $a_n\neq 0,b_m\neq 0$, and let $w\in \mathbb{C}$ be
given. Since $a_n\neq 0,b_m\neq 0$, the FTA  and a simple
manipulation imply that there exists $t\neq 0$ such that
$p(t)+q(1/t)=w$. Because $g(\mathbb{C})=\mathbb{C}\backslash\{0\}$
by Picard's theorem, there exists $z\in \mathbb{C}$ such that
$g(z)=t$, so $f(z)=w$ and $f$ is surjective. Hence, for instance,
the function $f(z)=\cos^7(z^2+z+i)-3\cos(z^2+z+i)+1$ is
surjective, and in particular it has a root.\\\\
{\noindent }\textbf{Acknowledgments}\vspace{0.2cm}\\ I would like to
thank Simeon Reich and Robert B. Burckel for some useful remarks and for making me aware of references \cite{Cater} and \cite{ReichAlgebra}.


\bibliographystyle{amsplain}
\bibliography{biblio}


\end{document}